\documentclass[11pt]{amsart}

\usepackage{amsmath}
\usepackage{amssymb}
\usepackage[all]{xy}
\newtheorem{thm}{Theorem}[section]

\newtheorem{dfn}[thm]{Definition}

\newtheorem{lem}[thm]{Lemma}

\newtheorem{prop}[thm]{Proposition}

\newtheorem{remarks}[thm]{Remark}

\newtheorem{cor}[thm]{Corollary}

\newtheorem{ex}[thm]{Example}

\def\bt{\begin{thm}}
\def\bp{\begin{prop}}
\def\blem{\begin{lem}}
\def\bd{\begin{dfn}}
\def\br{\begin{remarks}}
\def\bc{\begin{cor}}
\def\bex{\begin{ex}}
\def\beqs{\begin{eqnarray*}}
\def\beq{\begin{eqnarray}}
\def\bi{\begin{itemize}}

\def\et{\end{thm}}
\def\ep{\end{prop}}
\def\elem{\end{lem}}
\def\ed{\end{dfn}}
\def\er{\end{remarks}}
\def\ec{\end{cor}}
\def\eex{\end{ex}}
\def\eeqs{\end{eqnarray*}}
\def\eeq{\end{eqnarray}}
\def\ei{\end{itemize}}

\def\endpf{{\hfill$\square$\medskip}}

\def \ao{A_{\gamma_0}}

\def\wot{\widehat{\otimes}}
\def\D{\tilde{D}}

\def\td{\theta_{\delta}}

\def\ad{{\rm ad}}
\def\id{{\rm id}}
\def\c{\cdot}
\def\lo{L^1(G)}

\def\li{L^{\infty}(G)}

\def\r{\rangle}
\def\l{\langle}

\def\sf{S_0(G)}

\def\N{{\mathbb N}}
\def\C{{\mathbb C}}
\def\R{{\mathbb R}}

\def\sa{S^1\!A(G)}
\def\sah{S^1\!A(H)}
\def\NA{\|_{A(G)}}
\def\NSA{\|_{S^1\!A}}
\def\NL{\|_{L^1}}
\def\so{S^1(G)}

\begin{document}
\title[Biflatness and Pseudo-Amenability of  Segal algebras]
{Biflatness and Pseudo-Amenability of  Segal algebras}
\author[Ebrahim Samei,  Nico Spronk and
Ross Stokke]{Ebrahim Samei*,  Nico Spronk$^{**}$  and
Ross Stokke$^{***}$}

\maketitle

\begin{abstract}
We investigate generalized amenability and biflatness properties  of
various (operator) Segal algebras in both the group algebra, $\lo$,
and the Fourier algebra, $A(G)$, of a locally compact group, $G$.
\end{abstract}

\footnote{{\it Date}: \today.

* Research supported by an NSERC Post Doctoral Fellowship.
** Research supported by NSERC under grant no.\ 312515-05.
*** Research supported by NSERC under grant no.\ 298444-04.
}

Barry Johnson introduced the important concept of amenability for
Banach algebras in \cite{Joh1}, where he proved, among many other
things, that a group algebra $\lo$ is  amenable  precisely when the
locally compact group, $G$,  is amenable. For other Banach algebras,
it is often useful to relax some of the conditions in the original
definition of amenability and a popular theme in abstract harmonic
analysis has been to find, for various classes of Banach algebras
associated to  locally compact groups,  a ``correct
 notion'' of amenability in the sense that it singles out the
amenable groups.  For example, the measure algebra $M(G)$ is
amenable if and only if $G$ is both amenable and discrete
\cite{DGH}, however $M(G)$ is Connes-amenable (a definition of
amenability for dual Banach algebras) exactly when $G$ is amenable
\cite{Run2}.  As another example, the Fourier algebra, $A(G)$, can
fail to be amenable even for compact groups \cite{Joh2}, but is
operator amenable (a version of amenability that makes sense for
Banach algebras with an operator space structure) if and only if $G$
is amenable \cite{Rua}.

The purpose of this paper is to examine the amenability properties
of Segal algebras,  in both $\lo$ and $A(G)$. All of the
aforementioned versions of amenability imply the existence of a
bounded approximate identity (or identity in the case of
Connes-amenability), however, a proper Segal algebra never has a
bounded approximate identity \cite{Bur}.  Ghahramani, Loy and Zhang
have introduced several notions of ``amenablility without
boundedness", including approximate and pseudo-amenability, which
do not a priori imply the existence of bounded approximate identities
\cite{GL}, \cite{GZ}.   It is thus  natural to try to determine when
a Segal algebra is  approximately/pseudo-amenable. Indeed, this has
already been considered  in \cite{GL} and \cite{GZ}. In particular,
Ghahramani and Zhang showed that  if $S^1(G)$ is a Segal algebra in
$\lo$ with an approximate identity which ``approximately commutes
with orbits'' (this includes all [SIN]-groups) and $G$ is amenable,
then $\so$ is pseudo-amenable and that when $G$ is compact, $\so$ is
pseudo-contractible \cite[Propostion 4.4 and Theorem 4.5]{GZ} (also
see \cite[Corollary 7.1]{GL}). At present, there is  no known
example of an approximately amenable Banach algebra without a
bounded approximate identity, so in our study of Segal algebras we
will only consider pseudo-amenability and pseudo-contractibility. We
note that the approximate and pseudo-amenability of $\lo$, $M(G)$,
and $A(G)$ are studied in \cite{GL}, \cite{GZ} and \cite{GS}.

An important property, which is  related to amenability, is  the
homological notion of  biflatness (see, for example, \cite[Theorem
2.9.65]{Dal}).
%Recently, A. Yu. Pirkovskii characterized the
%biflatness of a Banach algebra, $A$, in terms of the existence of a
%type of uniform approximate left inverse inverse to $\pi^*$, where
%$\pi: A \wot A \rightarrow A$ is the multiplication operator.
In Section 2 we provide a natural generalization of biflatness, in
the spirit of the definitions of approximate and pseudo-amenability:
approximate biflatness. Our definition is inspired by a recent
characterization of biflatness of A.Yu.\ Pirkovskii \cite{Pir}.  We
prove  that a sufficient condition for $A$ to be pseudo-amenable is
that it is approximately biflat and has an approximate identity
(Theorem \ref{ApproxBiflat-PseudoProp}). The section concludes with
an examination of some hereditary properties of (approximately)
biflat Banach algebras that are needed in our study of the
approximate cohomology of Segal algebras.

In Section 3, we study Segal algebras, $\so$, in $\lo$. We prove
that $G$ is amenable when  $\so$ is pseudo-amenable (Theorem
\ref{SOPseudoAmThm}) and prove that for [SIN]-groups, $\so$ is
either pseudo-amenable or approximately biflat if and only if $G$ is
amenable.  For symmetric Segal algebras,  we show  that   $G$ is
amenable exactly when $\so$ is a flat $\lo$-bimodule; which happens
exactly when $\so$ has a type of approximate diagonal in $\lo
\widehat{\otimes} \so$ (Theorem \ref{T:pseudo-biflat-approx
diagonal}). This idea in then used in Theorem \ref{GhahLauThm} to
give an alternative approach to that of \cite{GLau} for describing
continuous derivations from $\so$ into $\lo$-modules when $G$ is
amenable. We show in Theorem \ref{CpctVsPseudoContrThm} that $\so$
is compact when $\so$ is pseudo-contractible (the converse to
\cite[Theorem 4.5]{GZ}).  Finally, in Theorem \ref{wstarapproxinner}
we prove, for any group $G$ and every continuous derivation $D:
S^1(G) \rightarrow S^1(G)^*$, that $\pi^* \circ D $ is
w*-approximately inner.

In Sections 4 and 5 we turn our attention to (operator) Segal
algebras in $A(G)$.  We first show in Theorem
\ref{SA(G)PseudoContractThm}  that an arbitrary Segal algebra
$SA(G)$ in $A(G)$ is pseudo-contractible if and only if $G$ is
discrete and $SA(G)$ has an approximate identity. We then focus on
the Lebesgue-Fourier algebra $S^1A(G) = A(G) \cap \lo$ which was
introduced by Ghahramani and Lau in  \cite{GL1}, and was recently
studied by Forrest, Wood, and the second-named author in \cite{FSW}.
As well, we will examine Feichtinger's Segal algebra $S_0(G)$ which
was shown by the second-named author to have many remarkable
properties \cite{Spr2}. We prove in Theorem \ref{S1APseudoThm} that
when $S^1A(G)$ has an approximate identity and $G$ contains an
abelian open subgroup, then $\sa$ is approximately biflat (and
therefore pseudo-amenable). Supposing that $G$ contains an open
subgroup, $H$, that is weakly amenable and such that $\Delta_H$, the
diagonal subgroup of $H \times H$,  has a bounded approximate
indicator (this is true for example whenever $G_e$, the connected
component of the identity,  is amenable), then $\sa$ is operator
approximately biflat (and therefore operator pseudo-amenable)
whenever it has an approximate identity (Theorem
\ref{S1AOpPseudoThm}). We conclude with Theorem \ref{S_0biflatCor}
which  shows that under these same hypotheses, the Feichtinger Segal
algebra, $S_0(G)$,  is actually operator biflat. This, in
particular, implies that it is operator pseudo-amenable.

\section{Preliminaries}

\subsection{Banach algebras of harmonc analysis}
Let $G$ be a locally compact group and let $M(G)$ be the Banach
space of complex-valued, regular Borel measures on $G$. The space
$M(G)$ is a unital Banach algebra with the convolution
multiplication and $L^1(G)$, the group algebra on $G$, is a closed
ideal in $M(G)$. We write $\delta_s$ for the point mass at $s\in G$;
the element $\delta_e$ is the identity of $M(G)$, and $l^1(G)$ is
the closed subalgebra of $M(G)$ generated by the point masses.

Let $G$ be a locally compact group,  let $P(G)$ be the set of all
continuous positive definite functions on $G$, and let $B(G)$ be its
linear span. The space $B(G)$ can be identified with the dual of the
group $C^*$-algebra $C^*(G)$, this latter being the completion of
$L^1(G)$ under its largest $C^*$-norm. With pointwise multiplication
and the dual norm, $B(G)$ is a commutative regular semisimple Banach
algebra. The Fourier algebra $A(G)$ is the closure of $B(G)\cap
C_c(G)$ in $B(G)$. It is shown in \cite{Eym} that $A(G)$ is a
commutative regular semisimple Banach algebra whose carrier space is
$G$. Also, up to isomorphism, $A(G)$ is the unique predual of
$VN(G)$, the von Neumann algebra generated by the left regular
representation of $G$ on $L^2(G)$.

Let $H$ be a closed subgroup of  $G$, and let $I(H) = \{ v \in A(G):
v \big{|}_H = 0 \}.$   A net $(u_\gamma)$ in $B(G)$ is called an \it
approximate indicator \rm for $H$ if
\begin{itemize}
\item[(i)] $\lim v (u_\gamma \big{|}_H) = v$ for all $v \in A(H)$;
and  \item[(ii)] $ \lim w u_\gamma  = 0 $ for all $w \in I(H)$.
\end{itemize}
Approximate indicators were introduced in \cite{ARS}.

\subsection{Operator spaces}  Our standard reference for operator spaces
is \cite{ER}.  We summarize some basic definitions, below.

Let $V$ be a Banach space.  An {\it operator space structure} on $V$
is a family of norms $\{\|\cdot\|_n:M_n(V)\to\R^{\geq
0}\}_{n\in\N}$ -- where each $M_n(V)$ is the space of $n\times n$
matrices with entries in $V$ -- which satisfy Ruan's axioms.  The
natural morphisms between operator spaces are the {\it completely
bounded maps}, i.e.\ those linear maps $T:V\to W$ which satisfy
$\|T\|_{cb}=\sup_{n\in\N}\|T_n\| < \infty$ where $T_n:M_n(V)\to
M_n(W)$ is given by $T_n[v_{ij}]=[Tv_{ij}]$.  We say that $T$ is
{\it completely contractive} if $\|T\|_{cb}\leq 1$.  Operator spaces
admit an analogue of the projective tensor product $\wot$, which we
call the operator projective tensor product $\wot_{op}$.

If $A$ is a Banach algebra which is also an operator space, and $V$
is a left $A$-module and an operator space, we say that $V$ is a
{\it completely contractive left $A$-module} if the product map
$\pi_0:A\otimes V\to V$ extends to a complete contraction
$\pi:A\wot_{op}V\to V$.  Completely contractive right and bi-modules
are defined similarly.  We say that $A$ is a completely contractive
Banach algebra if it is a completely contractive bimodule over
itself.  Natural examples include $\lo$, which inherits the maximal
operator space structure as the predual of a commutative von Neumann
algebra; and $A(G)$, which inherits its operator space structure as
the predual of $VN(G)$.

\subsection{Amenability properties}
Let $A$ be a (completely contractive) Banach algebra.

Following Johnson \cite{Joh} we say that $A$ is {\it (operator)
amenable} if $A$ admits is a {\it bounded approximate diagonal},
i.e.\ a bounded net $(m_\alpha)$ in $A\wot A$ (resp.\ in
$A\wot_{op}A$) such that
\begin{equation}\label{diagconds}
a\cdot m_\alpha-m_\alpha\cdot a\to0,\qquad \pi(m_\alpha)a \to a
\end{equation}
for all $a\in A$, where $a\cdot (b\otimes c)=(ab)\otimes c,\;
(b\otimes c)\cdot a=b\otimes(ca)$, and $\pi$ is the product map.
(Operator) amenability of $A$ is equivalent to having every
(completely) bounded derivation from $A$ into a(n operator) dual
bimodule be inner; see \cite{Joh}. A natural relaxation of
amenability is to allow $A$ to admit a diagonal net, as in
(\ref{diagconds}) above, but not insist that it is bounded.  In
doing so we obtain {\it (operator) pseudo-amenability}, as defined
in \cite{GL}. If $A$ admits a net in $A\wot A$ (resp.\ in
$A\wot_{op}A$) which satisfies (\ref{diagconds}), and the additional
property that $a\cdot m_\alpha=m_\alpha\cdot a$, then $A$ is said to
be {\it (operator) pseudo-contractible}, as defined in \cite{GZ}.

We say that $A$ is \it (operator) biflat \rm if there is a (completely bounded) 
bounded $A$-bimodule map $\theta:(A\wot A)^*\to A^*$ (resp.\ $(A\wot_{op}
A)^*\to A^*$) such that $\theta\circ\pi^*=\id_{A^*}$.  A.Ya.\
Helemskii proved that $A$ is amenable if and only if $A$ is biflat
and admits a bounded approximate identity; see \cite{Hel} or
\cite{CL}. The analogous characterization of operator amenability
follows similarly.  A (completely contractive) left $A$-module is
said to be {\it (operator) projective} if there is a (completely bounded) 
bounded left $A$-module
map $\xi:V\to A\wot V$ (resp.\ $V\to A\wot_{op} V$) such that
$\pi\circ\xi=\id_{V}$.  As similar definition holds for right
modules. $A$ is {\it (operator) biprojective} if there is a (completely bounded) 
bounded $A$-bimodule map $\xi:A\to A\wot A$ (resp.\ $A\to A\wot_{op} A$)
such that $\pi\circ\xi=\id_{A}$.% *******IS THIS CORRECT?***********

\subsection{Segal algebras}
Segal algebras were first defined by H.\ Reiter for group algebras;
see \cite{RS}, for example. The definition of operator Segal
algebras appeared in \cite{FSW}. However, our abstract definition
deviates from the one given in \cite{FSW} in the sense that we
demand that Segal algebras be essential modules.

Let $A$ be a (completely contractive) Banach algebra.  An {\it (operator)
Segal algebra} is a subspace $B$ of $A$ such that

(i) $B$ is dense in $A$,

(ii) $B$ is a left ideal in $A$, and

(iii) \parbox[t]{4.4in}{$B$ admits a norm (operator space structure)
$\|\cdot\|_B$ under which it is complete and a (completely)
contractive $A$-module.}

(iv) $B$ is an essential $A$-module:  $A\cdot B$ is $\|\cdot\|_B$-dense
in $B$.

\noindent We further say that $B$ is {\it symmetric}
if it is also a (completely) contractive essential right $A$-module.

In the case that $A=\lo$ we will write $\so$ instead of $B$ and further
insist that

(iv) \parbox[t]{4.4in}{$\so$ is closed under left translations:  $L_xf\in\so$ for
all $x$ in $G$ and $f$ in $\so$}

\noindent where $L_xf(y)=f(x^{-1}y)$ for $y$ in $G$.
By well-known techniques, condition (iii) on $B=\so$ is equivalent to

(iii') \parbox[t]{4.4in}{the map $(x,f)\mapsto L_xf:G\times\so\to\so$ is continuous
with$\|L_xf\|_{S^1}=\|f\|_{S^1}$ for all $x$ in $G$ and $f$ in $\so$.}

\noindent Moreover, symmetry for $\so$ is equivalent to having that
$\so$ is closed under right actions -- $R_xf\in\so$ for $x$ in $G$
and $f$ in $\so$, where $R_xf(y)=f(yx^{-1})\Delta(x^{-1})$ --  with
the actions being continuous and isometric.

We will discuss two specific types of operator Segal algebras in the
Fourier algebra $A(G)$.  One is the Lebesgue-Fourier algebra, $\sa$,
whose study was initiated in \cite{GL1}, and which was shown to be
an operator Segal algebra in \cite{FSW}.  The second is
Feichtinger's algebra $S_0(G)$, whose study in the non-commutaive
case was taken up in \cite{Spr2}; this study included an exposition
of the operator space structure.  Though slightly different
terminology was used in that article, it was proved there that
$S_0(G)$ is an operator Segal algebra in $A(G)$, in the sense
defined above.

\section{Approximate biflatness and pseudo-amenability}

Throughout this section, $A$ is a Banach algebra. Recall that if
$E$, $F$ are Banach spaces, then the weak$^*$ operator topology
(W*OT) on ${\mathcal B}(E, F^*)$ is the locally convex topology
determined by the seminorms $\{p_{e,f} : e \in E, \ f\in F\}$ where
$p_{e, f}(T) = | \l f, Te \r |$. On bounded sets, the W*OT is
exactly the $w^*$-topology of ${\mathcal B}(E, F^*)$ when identified
with $(E \wot F)^*$, so closed balls of ${\mathcal B}(E, F^*)$ are
W$^*$OT compact. When $E$ and $F$ are operator spaces,  ${\mathcal C
B}(E, F^*)$ is identified with $(E \wot_{op} F)^*$ \cite[Corollary
7.1.5]{ER}. On $\|\cdot\|_{cb}$-bounded subsets of ${\mathcal C
B}(E, F^*)$, the W*OT agrees with the weak* topology.

Suppose  that $X$ and $Y$ are Banach $A$-bimodules. Following A.Yu.
Pirkovskii \cite{Pir}, a net $(\td)_\delta$ of bounded linear maps
from $X$ into $Y$, satisfying
\beq \label{ApproxMorphismEq} \|\td(a
\c x ) - a \c \td(x) \| \rightarrow 0  \ \ {\rm and } \ \ \| \td(
x\c a ) - \td (x) \c a \| \rightarrow 0
\eeq for all $a$ in $A$, will be
called an \it approximate $A$-bimodule morphism \rm from $X$ to $Y$.
If  $Y$ is a  dual Banach space, and instead of norm convergence we
have $w^*$-convergence in (\ref{ApproxMorphismEq}), we call
$(\td)_\delta$ a \it $w^*$-approximate $A$-bimodule morphism. \rm

The following proposition may be compared with \cite[Corollary
3.2]{Pir}.

\bp \label{BiflatnessProp} The following statements are equivalent:
\bi
\item[(i)] $A$ is biflat;
\item[(ii)] there is a net $\td: (A \wot A )^* \rightarrow A^* \ (\delta \in \Delta)$
of $A$-bimodule morphisms such that $(\td)_\delta$ is uniformly
bounded in ${\mathcal B} ((A \wot A )^*,  A^*) $ and
$W^*OT\text{-}\!\lim_\delta \td \circ \pi^* = id_{A^*}$;
\item[(iii)] there is a $w^*$-approximate $A$-bimodule morphism
 $\td: (A \wot A )^* \rightarrow A^* \ (\delta \in \Delta)$ such that $(\td)_\delta$ is uniformly bounded in ${\mathcal B} ((A \wot
A )^*,  A^*) $ and  $W^*OT\text{-}\!\lim_\delta \td \circ \pi^* =
id_{A^*}$. \ei \ep

\begin{proof} The implications $(i) \Rightarrow (ii)$ and $(ii)
\Rightarrow (iii)$ are trivial.  Let   $(\td)_\delta$ be a
$w^*$-approximate morphism satisfying the properties of statement
$(iii)$. As bounded subsets of ${\mathcal B} ((A \wot A )^*,  A^*)$ are
relatively $W^*OT$ compact, $(\td)_\delta$ has a $W^*OT$ limit
point, $\theta$; we may assume that $ W^*OT\text{-}\!\lim_\delta \td
= \theta$. Routine calculations show that $\theta$ is an
$A$-bimodule map such that $\theta \circ \pi^* = id_{A^*}$.
\endpf\end{proof}

\br  \label{OpBiflatRemark} \rm When $A$ is a quantized Banach
algebra, one can similarly
 prove an operator space version of Proposition
 \ref{BiflatnessProp}: \bi  \item[]  \it $A$ is operator biflat if and only if
there is a net $\td: (A \wot_{op} A )^* \rightarrow A^* \ (\delta
\in \Delta)$ of completely bounded $A$-bimodule morphisms such that
$\sup_\delta \| \td \|_{cb} < \infty$ and
$W^*OT\text{-}\!\lim_\delta \td \circ \pi^* = id_{A^*}$. \rm \ei \er

 By dropping the condition of uniform boundedness from statement
(ii)
 of Proposition \ref{BiflatnessProp},  we obtain our definition of
 (operator) approximate biflatness. Remark
 \ref{ApproxBiflatNotBiflatRem} gives examples of approximately
 biflat Banach algebras which are not biflat.

 \bd  \rm \label{approx biflat def}  We call a (quantized) Banach algebra, $A$,
  \it (operator) approximately biflat \rm if
  there is a net $\td: (A \wot A )^* \rightarrow A^*$ (respectively,
   $\td: (A \wot_{op} A )^* \rightarrow A^*$)
   $  (\delta \in \Delta)$
of  (completely) bounded $A$-bimodule morphisms such that
$W^*OT\text{-}\!\lim_\delta \td \circ \pi^* = id_{A^*}$.  \ed

 Note that statement (iii) in the following theorem agrees with statement (iii)
 of Proposition \ref{BiflatnessProp}, except that we have dropped
 the condition  of uniform boundedness.  Statement  (ii) may be seen
 as an approximate biprojectivity condition.

\bt \label{ApproxBiflat-PseudoProp} Consider the following
conditions for a Banach algebra $A$: \bi
\item[(i)] $A$ is pseudo-amenable;
\item[(ii)] there is an approximate $A$-bimodule morphism
$(\beta_\delta)$ from $A$ into $A \wot A$ such that $$ \| \pi \circ
\beta_\delta (a) - a \| \rightarrow 0 \qquad (a \in A);$$
\item[(iii)] there is a $w^*$-approximate $A$-bimodule morphism
 $\td: (A \wot A )^* \rightarrow A^* \ (\delta \in \Delta)$ such that
  $W^*OT\text{-}\!\lim_\delta \td \circ \pi^* = id_{A^*}$;

 \item[(iv)] $A$ is approximately biflat. \ei
Then $(i) \Rightarrow (ii) \Rightarrow (iii)$ and  if $A$ has a
central approximate identity, then $(iii) \Rightarrow (i)$.    If
$A$ has an approximate identity, then $(iv) \Rightarrow (i)$. \et

\begin{proof} Assuming that condition (i) holds, let
$(m_\delta)$ be an approximate diagonal for $A$. Then it is easy to
check that   $$\beta_\delta: A \rightarrow A \wot A : a \mapsto a \c
m_\delta$$ satisfies the properties of condition (ii).  The dual
maps $\td = \beta_\delta^*$ satisfy the conditions of statement
(iii).

Suppose that $\td: (A \wot A )^* \rightarrow A^* \ (\delta \in
\Delta)$ satisfies the conditions of statement (iii) and let
$(e_\lambda)_{\lambda \in \Lambda}$ be a central approximate
identity for $A$. Then for any $a  \in A$ and $\psi \in (A \wot
A)^*$
\begin{align*} \lim_\lambda \lim_\delta &\l \psi,  a \c
\td^*(e_\lambda) - \td^*(e_\lambda) \c a \r   =  \lim_\lambda
\lim_\delta \l e_\lambda,   \td( \psi \c a )  - \td ( a \c  \psi ) \r \\
&  =   \lim_\lambda \lim_\delta \l e_\lambda,   \td( \psi \c a )  -
\td (\psi) \c a + \td (\psi) \c a -  \td ( a \c  \psi ) \r \\
(*) & =    \lim_\lambda \lim_\delta \l e_\lambda,   \td( \psi \c a
)  - \td (\psi) \c a \r  + \l e_\lambda,  a \c  \td (\psi)  -  \td (
a \c \psi ) \r  \\
 & =   \lim_\lambda ( 0 + 0) = 0
\end{align*}
where we have used the
centrality of $(e_\lambda)$ at line $(*)$. Also, for $a \in A$ and
$\phi \in A^*$, \beqs \lim_\lambda \lim_\delta \l \phi,
\pi^{**}(\td^*(e_\lambda)) \c a \r & = &  \lim_\lambda \lim_\delta
\l e_\lambda , \td (\pi^* ( a \c  \phi))  \r \\
&  = &    \lim_\lambda \l e_\lambda ,  a \c  \phi  \r  =
\lim_\lambda \l e_\lambda a ,  \phi \r  \\
&  = &  \l a, \phi \r.  \eeqs
 Let $E= \Lambda \times \Delta^\Lambda$
be directed by the product ordering and for each $\beta = (\lambda,
(\delta_{\lambda'})) \in E$, let $m_\beta =
\theta_{\delta_\lambda}(e_\lambda) \in (A \wot A)^{**}$. Using the
iterated limit theorem \cite[p. 69]{Kel}, the above calculations
give for each $a$ in $A$
\begin{align} \label{DualApproxDiagEq}
a \c m_\beta - m_\beta \c a\rightarrow 0, \ w^* \ {\rm in } \ (A \wot A)^{**}  \\
{\rm and } \; \pi^{**}(m_\beta) a \rightarrow a, \ w^* \ {\rm in } \ A^{**}. \notag
\end{align}
 As in the proof of \cite[Proposition 2.3]{GZ} we can use
 Goldstine's theorem to obtain  $(m_\beta) $ in $A \wot A$, and we
 can replace weak$^*$ convergence in equation (\ref{DualApproxDiagEq}) by weak convergence.  This
 implies, via Mazur's theorem, that $A$ is pseudo-amenable (again
 see \cite[Propostion 2.3]{GZ}).

 The proof that $A$ is pseudo-amenable when $A$ is approximately biflat and has an approximate
 identity  $(e_\lambda)$ is the same as that given above, except that we reverse
 the order in which we calculate the iterated limits and use the
 fact that  each $\td$ is  now an $A$-bimodule map:
\beqs  \lim_\delta \lim_\lambda    a \c \td^*(e_\lambda) -
\td^*(e_\lambda) \c a  =  \lim_\delta \lim_\lambda   \td^*( a \c
e_\lambda  - e_\lambda  \c a) =  \lim_\delta 0 = 0 \eeqs and
\begin{align*}
 \lim_\delta  \lim_\lambda \l \phi, \pi^{**}(\td^*(e_\lambda)) \c a \r
& =   \lim_\delta \lim_\lambda \l e_\lambda a , \td (\pi^* ( \phi))
\r \\ &=    \lim_\delta  \l a ,  \td (\pi^* ( \phi) )  \r  =
 \l a, \phi \r.
 \end{align*}
This completes the proof. \endpf\end{proof}

One can similarly prove the analogous relationship between operator
pseudo-amenability and operator approximate biflatness.   Our
motivation in writing this paper has been to obtain information
about the approximate (co)homology of Segal algebras, so we will not
attempt to exhaustively determine the relationship between
approximate biflatness and other forms of amenability.   Instead, we
have  chosen  to only examine approximate biflatness versus
pseudo-amenability  (Theorem \ref{ApproxBiflat-PseudoProp}) and
refer the reader to
 \cite{GL} for a detailed
study of the relationship between pseudo-amenability and several
other amenability properties.  We will,  however, conclude this
section with an examination of some hereditary properties of
(approximately) biflat Banach algebras that are needed in the
sequel.

\bp \label{AbstractApproxBiflatHeredProp} Let $B$ be an (operator)
Segal algebra in $A$, and suppose that $B$ contains a net
$(e_\lambda)_{\lambda \in \Lambda}$ in its centre such that
$(e_\lambda^2)_{\lambda \in \Lambda}$ is an approximate identity for
$B$.
  If $A$ is
(operator) approximately biflat, then so is $B$. \ep

\begin{proof}  We will prove the operator space version
of the proposition -- the other case is  similar.   Let
 $T_\lambda$ be the completely bounded
 map specified by
$$T_\lambda : A \wot_{op} A \rightarrow B \wot_{op} B : a \otimes b \mapsto
 a e_\lambda \otimes b
e_\lambda.$$ As $e_\lambda$ is central in $B$, $T_\lambda$ is a
$B$-bimodule map. Let $\td: (A \wot A )^* \rightarrow A^* \ (\delta
\in \Delta)$ be a net of completely bounded $A$-bimodule maps such
that $W^*OT\text{-}\!\lim_\delta \td \circ \pi_A^* = id_{A^*}$, and
consider the completely bounded $B$-bimodule map,  $p: A^*
\rightarrow B^*: \phi \mapsto \phi \Big{|}_B$. Let $E = \Lambda
\times \Delta^\Lambda$ be directed by the product ordering, and for
each $\beta = (\lambda, (\delta_{\lambda'})_{\lambda'} ) \in E$,
define $\theta_\beta:(B \wot_{op} B)^* \rightarrow B^*$ so that the
diagram commutes:
\[
\xymatrix{ (A \wot_{op} A)^*
\ar@<.5ex>[rr]^{\theta_{\delta_\lambda}} & & A^* \ar[d]^{p}
  \\
(B \wot_{op} B)^* \ar[u]^{T_\lambda^*} \ar@<.5ex>[rr]^{\theta_\beta}
& & B^* }
\]
That is, $\theta_\beta = p \circ  \theta_{\delta_\lambda} \circ
T_\lambda^*$, a completely bounded $B$-bimodule map. Note that
because $e_\lambda$ lies in the centre of $B$,
$$T_\lambda^* \circ \pi_B^* (\phi) = \pi_A^* \circ R_\lambda^*
(\phi) \qquad (\lambda \in \Lambda, \ \phi \in B^*),$$ where
$R_\lambda:A\rightarrow B: a \mapsto a e_\lambda^2 $.   Let $\phi
\in B^*$, $b \in B$.  By the iterated limit theorem we have \beqs
\lim_\beta \l b, \theta_\beta \circ \pi_B^* (\phi) \r & = &
\lim_\lambda \lim_\delta \l b, (p \circ
\theta_{\delta} \circ T_\lambda^* \circ \pi_B^*) (\phi) \r \\
& = & \lim_\lambda \lim_\delta \l b,  (\theta_{\delta} \circ \pi_A^*
\circ R_\lambda^*) (\phi) \r \\
& = & \lim_\lambda  \l b,
 R_\lambda^* (\phi) \r \\
& = & \lim_\lambda  \l  b e_\lambda^2 ,
 \phi \r  \\
 & = & \l b, \phi \r.
\eeqs Hence, $W^*OT\text{-}\!\lim_\beta \theta_\beta \circ \pi_{B^*}
= id_{B^*}$.  \endpf\end{proof}

Note that if $(e_\lambda)_\lambda$ is an approximate identity which
is bounded in the multiplier norm on $B$, then
$(e_\lambda^2)_\lambda$ is also an approximate identity for $B$.

\bd \rm The   \it (operator) biflatness constant \rm  of an
(operator) biflat  (quantized) Banach algebra $A$ is  the number
$BF_A = \inf_\gamma \|\theta\|$ (respectively, $BF^{op}_A =
\inf_\gamma \|\theta\|_{cb}$) where the infimum is taken over all
(completely) bounded $A$-bimodule maps $\theta : (A \wot A)^*
\rightarrow A^*$ (resp. $\theta : (A \wot_{op} A)^* \rightarrow
A^*$) such that $\theta \circ \pi^* = id_{A^*}$. \ed

\bp \label{DirectedUnionProp} Let $A$ be a  (quantized) Banach
algebra containing a   directed family of closed ideals $\{
A_\gamma: \gamma \in \Gamma\}$ such that for each $\gamma \in
\Gamma$ there is a (completely) bounded homomorphic projection
$P_\gamma$ of $A$ onto $A_\gamma$. Suppose that either \bi
\item[(i)] $A$ has a central approximate identity
$(e_\lambda)_\lambda$ in $\cup_\gamma A_\gamma$; or
\item[(ii)] for each $a \in A$, $\| P_\gamma a - a \| \rightarrow
0$.

\bi
\item[(a)] If each $A_\gamma$ is (operator) approximately biflat, then so is $A$.
\item[(b)] If each $A_\gamma $ is (operator) biflat with $\sup_\gamma
BF_{A_\gamma}< \infty $ (respectively, $\sup_\gamma
BF_{A_\gamma}^{op} < \infty $), and  (i) holds with
$(e_\lambda)_\lambda$  bounded in the (completely bounded)
multiplier norm of $A$, or (ii) holds with $\sup_\gamma \|
P_\gamma\| < \infty$ (respectively, $\sup_\gamma \| P_\gamma\|_{cb}
< \infty$), then $A$ is (operator) biflat. \ei \ei \ep

\begin{proof} We first prove (a).   Given $\alpha = (F, \Phi, \epsilon)$ where $F \subset
A $, $\Phi \subset A^*$ are finite, and $\epsilon >0$, we will find
an $A$-bimodule map $\theta_\alpha: (A \wot A)^* \rightarrow A^* $
such that \beq \label{DirectedUnionEq0}|\l a, (\theta_\alpha \circ
\pi_A^* )(\phi) - \phi \r | < \epsilon  \qquad (a \in F, \ \phi \in
\Phi). \eeq Assuming first that condition (i) holds, take
$e_{\lambda_0} = e_0$ such that \beq \label{DirectedUnionEq1} \|a
e_0 - a \| < \epsilon /2M \qquad (a \in F) \eeq where $M = \sup \{
\| \phi \| : \phi  \in \Phi \}$. Choose  $\gamma_0 \in \Gamma$ such
that $e_0 \in A_{\gamma_0}$. Consider the maps $$ \iota_0 : \ao \wot
\ao \rightarrow A \wot A : a \otimes b \mapsto a \otimes b \ \ {\rm
and } \ \ T_0 : A \rightarrow \ao: a \mapsto a e_0$$ and let $\pi_0:
\ao \wot \ao \rightarrow \ao$ be the multiplication map.  As $\ao$
is approximately biflat, there is an $\ao$-bimodule map $\theta_0 :
(\ao \wot \ao)^* \rightarrow \ao^* $ such  that \beq
\label{DirectedUnionEq2}  |\l T_0a, (\theta_0 \circ \pi_0^* )(\phi
\big{|}_{\ao}) -  \phi \big{|}_{\ao} \r | < \epsilon/2  \qquad (a
\in F, \ \phi \in \Phi). \eeq Define $\theta_\alpha $ so that the
diagram commutes:
\[
\xymatrix{ (\ao \wot \ao)^* \ar@<.5ex>[rr]^{\theta_0} & & \ao^*
\ar[d]^{T_0^*}
  \\
(A \wot A)^* \ar[u]^{\iota_0^*} \ar@<.5ex>[rr]^{\theta_\alpha} & &
A^* }
\]
That is, let $\theta_\alpha = T_0^* \circ \theta_0 \circ \iota_0^*$.
For $a \in F$ and $\phi \in \Phi$, equations
(\ref{DirectedUnionEq1}) and (\ref{DirectedUnionEq2}) give \beqs |\l
a, (\theta_\alpha \circ \pi_A^* )(\phi) - \phi \r | & \leq & |\l
T_0a, (\theta_0 ( \iota_0^* ( \pi_A^* (\phi))) - \phi \r | + | \l
T_0a - a, \phi \r |
\\
& \leq & | \l T_0a,(\theta_0 \circ \pi_0^* )(\phi \big{|}_{\ao}) -
\phi \big{|}_{\ao} \r | + \| ae_0 - a\| \| \phi\|\\  & < & \epsilon.
 \eeqs

 If condition (ii) holds,  we instead choose $\gamma_0$ such that
  $\|P_{\gamma_0}a  - a \| < \epsilon /2M$ $(a \in
F)$.  By replacing $T_0$ in the above paragraph by $P_{\gamma_0}$ we
again obtain equation (\ref{DirectedUnionEq0}).

Because we only know that $\theta_0$ is an $\ao$-bimodule map, the
argument showing that that $\theta_\alpha$ is an $A$-bimodule map
requires some care. Note that \beqs \iota_0^*( a \c \psi) =
P_{\gamma_0} (a) \c \iota_0^* ( \psi) \qquad (a \in A, \  \psi \in
(A \wot A)^*) \eeqs where on the left and right we respectively have
  $A$-module, and $\ao$-module, actions. Let $a, b  \in A$, $\psi
\in (A \wot A)^*$ and assume first that $\theta_\alpha = T_0^* \circ
\theta_0 \circ \iota_0^*$. Then \beqs \l b, \theta_\alpha ( a \c
\psi) \r &
= & \l T_0 b , \theta_0 ( \iota_0^* (a \c \psi) ) \r \\
& = & \l T_0b  , \theta_0 ( P_{\gamma_0}(a) \c \iota_0^* ( \psi) )
\r \\
& = & \l T_0 b, P_{\gamma_0} (a) \c \theta_0 ( \iota_0^* ( \psi) )
\r \\
& = & \l  T_0 (b )P_{\gamma_0} (a), \theta_0 (  \iota_0^* (\psi) )
\r
\\& = & \l  T_0 (b  a), \theta_0 (  \iota_0^* (\psi) )
\r
\\
& = & \l b a, T_0^* (\theta_0 ( \iota_0^* (  \psi) ) ) \r \\
  & = & \l  b,  a \c \theta_\alpha  (  \psi)  \r,
 \eeqs where we have used the fact that $T_0 b P_{\gamma_0} a = P_{\gamma_0}((T_0 b) a)  =
 be_0 a = bae_0 = T_0(ba)$.  As well, $P_{\gamma_0} b P_{\gamma_0} a  = P_{\gamma_0}
 (ba)$, so the same argument works when $\theta_\alpha = P_{\gamma_0}^* \circ
\theta_0 \circ \iota_0^*$.
  A symmetric argument shows that
$\theta_\alpha$ is also a right $A$-module map. The operator
biflatness version of part (a) is proved in exactly the same way.

Under the hypotheses of the non-bracketed part of statement  (b),
the maps $\theta_\alpha$ can be chosen to be  uniformly bounded in
${\mathcal B} ((A \wot A)^*, A^*)$, so biflatness follows from
Proposition \ref{BiflatnessProp}.  If $A$ is a quantized Banach
algebra, then the bracketed hypotheses of statement (b) yield
completely bounded maps $\theta_\alpha$ in ${\mathcal CB} ((A \wot_{op}
A )^*, A^*) $ such that $\sup_\alpha \|\theta_\alpha\|_{cb} <
\infty$. Operator biflatness of $A$ follows from Remarks
\ref{OpBiflatRemark}.
\endpf\end{proof}

If $\{V_i:i\in I\}$ is a family of operator spaces, we let
$\bigoplus_{i\in I}^p V_i$ ($1\leq p<\infty$) have the operator
space structure it attains as the predual of the direct product of
dual spaces in the case $p=1$, and through interpolation in the case
$p>1$.  See \cite{Pis}.

  \bp \label{DirectSumProp} Let $\{ A_i: i \in I\} $ be a family of
 (quantized)
Banach algebras.  \bi \item[(i)] If each $A_i$ is (operator)
approximately biflat,  then for $1 \leq p < \infty$,
$\bigoplus_{i\in I}^p A_i$ is (operator) approximately biflat.

\item[(ii)] If $A_1$, $A_2$ are operator approximately biflat
quantized Banach algebras and $A= A_1 \oplus A_2$ has an operator
space structure such that the projection maps $A \rightarrow A_i$
are completely bounded, then $A$ is also operator approximately
biflat.

\item[(iii)]  If each $A_i$ is
(operator)  biflat and  $\sup_i BF_{A_i} < \infty$ (respectively
$\sup_i BF_{A_i}^{op}  < \infty$), then $\bigoplus_{i\in I}^1A_i$ is
(operator) biflat. \ei
 \ep

\begin{proof}  We first prove (ii).    Let $\alpha = (F, \Phi,
\epsilon)$, where $\epsilon > 0$, and $F \subset A$, $\Phi \subset
A^*$ are finite.  Let $\theta_i : (A_i \wot_{op} A_i)^* \rightarrow
A_i^*$ be a completely bounded  $A_i$-bimodule map such that
$$ |\l a_i, \theta_i \circ \pi_{A_i}^* (\phi_i) - \phi_i \r| <
\epsilon /2 \qquad ( i = 1,2, \ a= (a_1, a_2) \in F, \ \phi \in
\Phi)
$$ (where $\phi_i = \phi \Big{|}_{A_i}$).  Let $E_i : A_i \hookrightarrow A$
and $p_i : A \rightarrow A_i$ be the embedding and projection maps
and let $\widetilde{\theta_i} = p_i^* \circ \theta_i \circ(E_i
\otimes E_i)^* \ (i = 1,2)$. Thus, we have the commuting diagram:
\[
\xymatrix{ (A \wot_{op} A)^* \ar[d]_{(E_i \otimes
E_i)^*}\ar@<.5ex>[rr]^{\widetilde{\theta_i}} & & A^*
  \\
(A_i \wot_{op} A_i)^*  \ar@<.5ex>[rr]^{\theta_i} & & \ar[u]_{p_i^*}
A_i^* }
\]
Standard arguments show  that $\widetilde{\theta}=
\widetilde{\theta_1} + \widetilde{\theta_2} : (A \wot_{op} A)^*
\rightarrow A^* $ is a completely bounded $A$-bimodule map such that
$$ |\l a, \widetilde{\theta} \circ \pi_{A}^* (\phi) - \phi  \r| <
\epsilon \qquad (a \in F, \ \phi \in \Phi).
$$
This proves (ii). Obviously, the (non-quantized) Banach algebra
version of (ii) holds for arbitrary direct sums $A_1 \oplus A_2$.
Suppose further that  $A= A_1 \oplus^1 A_2$ is the (operator space)
$\ell^1$-direct sum of $A_1$ and $A_2$.  If each $A_i$ is (operator)
biflat and $\theta_i \circ \pi_{A_i}^* = id_{A_i^*}$, then observe
that $\widetilde{\theta} \circ \pi_{A}^* = id_{A^*}$ and
$\|\widetilde{\theta}\| \leq \max \{ \|\theta_1\|, \| \theta_2 \|
\}$ (respectively, $\|\widetilde{\theta}\|_{cb} \leq \max \{
\|\theta_1\|_{cb}, \| \theta_2 \|_{cb} \}$).

Suppose now that for each $ i \in I$, $A_i$ is (operator)
approximately biflat. Let $\Gamma = \{ \gamma : \gamma \subset I \
{\rm is \ finite} \}$ be ordered by inclusion.  By induction, the
first case shows that $A_\gamma = \bigoplus_{i\in \gamma}^pA_i$ is
(operator) approximately biflat. Viewing $A_\gamma$ as an ideal in
$A = \bigoplus_{i\in I}^pA_i$, the natural homomorphic projection
maps $P_\gamma$ of $A$ onto $A_\gamma$ are (completely) contractive
and satisfy $\| P_\gamma a - a \| \rightarrow 0$ $(a \in A)$. By
Proposition \ref{DirectedUnionProp}, $A$ is (operator) approximately
biflat. This is statement (i).

Finally, suppose that each $A_i$ is operator biflat with $\sup_{i\in
I}BF^{op}_{A_i} < \infty$.  As noted above,  $A_\gamma =
\bigoplus_{i\in \gamma}^1 A_i$ is operator biflat with
$BF_{A_\gamma}^{op} \leq \max_{i \in \gamma} BF_{A_i}^{op}$, so the
biflatness of $\bigoplus_{i\in I}^1A_i$ follows from Proposition
\ref{DirectedUnionProp}. This proves the operator space version of
(iii). The other case is similar.
\endpf\end{proof}

% Perhaps we should  state as a remark (without proof)
% some other hereditary properties of approx biflatness
% for eg homomorphic image results.

\section{Approximate biflatness and pseudo-amenability of $\so$}

Throughout this section, $\so$ will  denote an arbitrary  Segal
algebra in $\lo$, where $G$ is a locally compact group. Observe that
because $\so$ embeds contractively onto a dense subspace of $\lo$,
$\li$ in turn embeds contractively into $\so^*$ via
$$ \l f, \phi \r = \int_G f(s) \phi(s) \  ds \qquad (f \in \so, \phi \in \li).$$

\bt  \label{SOPseudoAmThm} If $\so$ is pseudo-amenable, then $G$ is
amenable. \et

\begin{proof}  Let $(m_\gamma)_{\gamma \in \Gamma} \subset \so \wot
\so$ be an approximate diagonal for $\so$. Let $\iota : \so
\hookrightarrow \lo$ be the embedding map, let $1_G$ be the
augmentation character of $L^1(G)$, and put $$T = \iota \otimes 1_G:
\so \wot \so \rightarrow \lo : f \otimes g \mapsto (\int_G g(s)ds )
f.$$  By checking with elementary tensors, one can see that $T$
satisfies $$ T(k \cdot m ) = k * Tm  \ \ {\rm and } \ \ T(m \c k ) =
(\int_G k(s)ds ) Tm $$ where $k  \in \so, \  m \in \so \wot \so$.
Hence, for any $ k \in \so$ with $ \int_G k(s)ds = 1$, we have
\begin{align} \label{SOPseudoAmThmEq1}
\| k* T m_\gamma - Tm_\gamma \|_{\lo} &= \| T ( k \c m_\gamma - m_\gamma
\c k) \|_{\lo} \\ & \leq \| k \c m_\gamma - m_\gamma \c k \| \rightarrow
0. \notag
\end{align}
Fix $h \in \so$ with $\int h = 1$, and for each $\gamma$
let $f_\gamma = h * Tm_\gamma$.  For each $x \in G$ we then obtain
\begin{align} \label{SOPseudoAmThmEq} \| \delta_x * f_\gamma - f_\gamma
\|_{\lo} &\leq \| (\delta_x*h) * T m_\gamma - Tm_\gamma \|_{\lo}  \\
&\qquad + \|T m_\gamma - h*Tm_\gamma \|_{\lo} \rightarrow 0. \notag
\end{align}
When $m =
f\otimes g$, note that $$\l 1_G, \pi (m) \r = \l 1_G, f *g\r =  \l
1_G, f \r\l 1_G, g\r = \l 1_G, \l 1_G, g \r f \r = \l 1_G, Tm\r, $$
and so \beqs 1 & = & \l 1_G, h \r  \\ &=& \lim_\gamma \l 1_G, h *
\pi(m_\gamma)\r \\ & =&  \lim_\gamma \l 1_G, h \r  \l 1_G,
\pi(m_\gamma)\r  \\ &= & \lim_\gamma \l 1_G, T m_\gamma \r \\ &=
&\lim_\gamma \l 1_G, h * T m_\gamma \r \\
& =&  \lim_\gamma \l 1_G, f_\gamma\r. \eeqs As $\|f_\gamma\|_{\lo}
\geq |  \l 1_G, f_\gamma\r|$ we may therefore assume that
$\|f_\gamma\|_{\lo} \geq 1/2$  $( \gamma \in \Gamma)$.  Defining
$$g_\gamma = \frac{1}{\|f_\gamma\|_{\lo}}|f_\gamma|, \qquad (\gamma
\in \Gamma), $$ we obtain a net of positive norm-one functions in
$\lo$ which by (\ref{SOPseudoAmThmEq1}) satisfies $$\| \delta_x *
g_\gamma - g_\gamma \|_{\lo} \leq 2 \| \delta_x * |f_\gamma| -
|f_\gamma| \|_{\lo} \leq 2 \| \delta_x * f_\gamma - f_\gamma
\|_{\lo} \longrightarrow 0$$ for $x \in G$. This implies that $G$ is
amenable \cite{Pat} -- any $w^*$-limit point of $(g_\gamma)_\gamma$
in $\li^*$ is a left-invariant mean on $\li$.
\endpf\end{proof}

\bc \label{SINCor}Let $G$ be a [SIN]-group.  Then  the following
statements are equivalent: \bi
\item[(i)] $G$ is amenable;
\item[(ii)] $\so$ is approximately biflat;
\item[(iii)] $\so$ is pseudo-amenable.  \ei \ec

\begin{proof}  If statement (i) holds, then $L^1(G)$ is amenable and therefore biflat
\cite[Theorem 2.9.65]{Dal}, and $S^1(G)$ has a central approximate
identity $(e_\lambda)_\lambda$ which is bounded in $L^1(G)$
\cite{KR}. Hence, $(e_\lambda^2)_\lambda$ is also an approximate
identity for $\so$, so (ii) is a consequence of Proposition
\ref{AbstractApproxBiflatHeredProp}. That (ii) implies (iii) and
(iii) implies (i) are special cases of Theorems
\ref{ApproxBiflat-PseudoProp} and  \ref{SOPseudoAmThm} respectively.
 \endpf\end{proof}

Proposition 4.4 of \cite{GZ} states that the converse to Theorem
\ref{SOPseudoAmThm} holds when $\so$ has an approximate identity
which ``approximately commutes with orbits''.  When $G$ is a
[SIN]-group, $\so$ always has such an approximate identity so $(i)
\Rightarrow (iii)$ of Corollary \ref{SINCor} is also a consequence
 \cite[Proposition 4.4]{GZ}.

We do not know whether, in general, the amenability of $G$ implies
either approximate bilfatness or pseduo-amenability of $\so$ (see
also \cite[Question 3, P. 123]{GZ}). However, as we show below, it
is possible to construct a well-behaved approximate diagonal for
$\so$ in $L^1(G)\widehat{\otimes} \so$ when $G$ is amenable.

We say that $\so$ has an approximate diagonal in
$L^1(G)\widehat{\otimes} \so$ if there is a net $\{ m_\gamma
\}_{\gamma \in \Gamma}$ in $L^1(G)\widehat{\otimes} \so$ such that,
for every $f\in \so$,
$$f\cdot m_\gamma-m_\gamma \cdot f \longrightarrow 0 \ \text{as} \ \gamma\longrightarrow \infty$$
and $\pi(m_\gamma)$ is an approximate identity for $\so$. If, in
addition, the associated left and right multiplication operators
$L_\gamma : f \mapsto f\cdot m_\gamma $ and $R_\gamma : f \mapsto
m_\gamma\cdot f$ from $\so$ into $L^1(G)\widehat{\otimes}\so$ are
uniformly bounded, then we say that $\so$ has a multiplier bounded
approximate diagonal in $L^1(G)\widehat{\otimes} \so$. Finally, in
either of the above cases, we say that the (multiplier-bounded)
approximate diagonal is central if $f\cdot m_\gamma=m_\gamma \cdot
f$ for all $\gamma \in \Gamma$ and $f\in \so$.

\bt \label{T:pseudo-biflat-approx diagonal} Let $G$ be a locally
compact group, and let $\so$ be a symmetric Segal algebra. Then the
following statements are equivalent:

\item[(i)] $G$ is amenable;
\item[(ii)] $\so$ is a flat $\lo$-bimodule;
\item[(iii)] $\so$ has an approximate diagonal in  $\lo \widehat{\otimes} \so$.
\item[(iv)] $\so$ has a multiplier-bounded approximate diagonal in  $\lo \widehat{\otimes} \so$.

\et

\begin{proof}
(i) $\Longrightarrow$ (ii) Since $G$ is amenable, $\lo$ is amenable.
Also $\so$ is an essential Banach $L^1(G)$-bimodule. Hence if
$\pi_1$ is the convolution multiplication map from $\lo
\widehat{\otimes} \so$ onto $\so$, then the short exact sequence of
$\lo$-bimodules
$$ 0 \longmapsto \so^* \stackrel{\pi_1^*}{\longrightarrow} (\lo \widehat{\otimes} \so)^* \stackrel{\iota^*}{\longrightarrow} (\ker \pi_1)^* \longmapsto 0,$$
is admissible, and  therefore splits \cite[Theorem 2.5]{CL}.

(ii) $\Longrightarrow$ (iv)  Let $\theta : (\lo \widehat{\otimes}
\so)^* \to \so^*$ be a continuous $\lo$-bimodule morphism such that
$\theta \circ \pi^*=id_{\so^*}$. Let $\{e_\alpha \}$ be an
approximate identity for $\so$ with $L^1$-norm equal to 1. Set
$n_\alpha=\theta^*(e_\alpha^2)\in (\lo \widehat{\otimes} \so)^{**}$.
Then, for every $f\in \so$ and  every $\alpha$,
$$\| f\cdot n_\alpha \|=\| \theta^*(f*e_\alpha^2) \| \leq \|\theta \| \|f*e_\alpha^2\|_{\so} \leq
\|\theta \| \|f \|_{\so}. $$ Similar to the above, we have $\|
n_\alpha\cdot f \| \leq \|\theta \| \|f\|_{\so}$. Also
$$\pi^{**}(n_\alpha)=(\pi^{**} \circ \theta^*)(e_\alpha^2)=e_\alpha^2,$$
which is an approximate identity for $\so$. Finally, for $f\in \so$
and $\varphi \in (\lo \widehat{\otimes} \so)^*$, we have \beqs \l
f\cdot n_\alpha -n_\alpha \cdot f \ , \ \varphi \r  &=& \l
\theta^*(e_\alpha^2) \ , \ \varphi \cdot f-f\cdot \varphi \r \\ & =&
\l \theta(\varphi) \cdot f-f\cdot \theta(\varphi) \ , \ e_\alpha^2
\r  \\ &= & \l \theta(\varphi) \ , \ f* e_\alpha^2-e_\alpha^2* f
\r. \eeqs Hence
$$\| f\cdot n_\alpha -n_\alpha \cdot f \| \leq \|\theta \| \| f* e_\alpha^2-e_\alpha^2* f \|_{\so} .$$
Therefore $f\cdot n_\alpha -n_\alpha \cdot f \rightarrow \infty$ as
$\alpha \rightarrow \infty$.
The final result follows from a similar argument to the one made in \cite[Proposition 2.3]{GZ}.\\
(iv) $\Longrightarrow$ (iii) is obvious and
   (iii) $\Longrightarrow$
(i) follows  the  argument found in the proof of  Theorem
\ref{SOPseudoAmThm}.
\endpf\end{proof}

The following is \cite[Theorem 3.1]{GLau}. Here we present an
alternative proof using the multiplier bounded approximate
diagonals.

\bt \label{GhahLauThm}
Let $G$ be a locally compact amenable group, let $\so$ be a
symmetric Segal algebra, and let $X$ be a Banach $\lo$-bimodule.
Then for every continuous derivation $D : \so \to X^*$, there is a
continuous double centralizer $(S,T)$ such that $D=S-T$.
%\item[(ii)] Every continuous derivation $D : \so \to X$ is approximately inner.
\et

\begin{proof}
Suppose that $D : \so \to X^{*}$ is a continuous derivation. By
applying the argument presented in the first two paragraphs of the
proof of \cite[Theorem 3.1(ii)]{GLau}, we can assume that $X$ is an
essential $\lo$-bimodule. By the proof of Theorem
\ref{T:pseudo-biflat-approx diagonal}(iv), we can choose a
multiplier-bounded approximate diagonal  $\{ m_\alpha \}$ for $\so$
in $\lo \widehat{\otimes} \so$ so that $\pi(m_\alpha)$ is bounded in
$L^1$-norm.

Let $m_\alpha=\sum_{i=1}^\infty f_i^\alpha \otimes g_i^\alpha$ and
define $x^*_\alpha=\sum_{i=1}^\infty f_i^\alpha \c D(g_i^\alpha)$.
Then, for $f\in \so$,
$$ f\cdot x^*_\alpha -x^*_\alpha \cdot f-\pi(m_\alpha)\c D(f)
\stackrel{\alpha}{\longrightarrow} 0. \eqno{(1)}$$ On the other
hand, the operators $S_\alpha : f \mapsto f\cdot x^*_\alpha $ and
$T_\alpha : f \mapsto x^*_\alpha\cdot f$ from $\so$ into $X^*$ are
uniformly bounded. Let $S$ be a cluster point of $\{S_\alpha\}$, and
let $T$ be a cluster point of $\{T_\alpha\}$ in the
weak$^*$-operator topology. Then $(S,T)$ is a double centralizer and
for every $f\in \so$ and $\xi \in X$,
$$ \l \xi, S(f)-T(f)- D(f) \r = \lim_\alpha \l \xi,  S_\alpha(f)
- T_\alpha (f) - \pi(m_\alpha)\c D(f) \r = 0,
$$ where we have used equation (1) and  the fact that $X$ is essential.
%Follows from part (i) with a similar argument in \cite[Theorem 3.1]{GLau}(ii).
\endpf\end{proof}

It is shown in \cite[Theorem 4.5]{GZ} that $\so$ is
pseudo-contractible if $G$ is compact. In the following theorem, we
prove the converse of that result and present other equivalent
conditions on pseudo-contractiblity of $\so$ (see also
\cite[Proposition 3.8]{GZ}).

\bt \label{CpctVsPseudoContrThm} Let $G$ be a locally compact group,
and let $\so$ be a Segal algebra. Then the following statements are
equivalent:

\item[(i)] $G$ is compact;
\item[(ii)] $\so$ has a central approximate diagonal in  $\lo \widehat{\otimes} \so$;
\item[(iii)] $\so$ is pseudo-contractible.

If, in addition, $\so$ is symmetric, then the above statements are
equivalent to either of the following statements:

\item[(iv)] $\so$ is a projective $\lo$-bimodule;
\item[(v)] $\so$ has a central multiplier-bounded approximate diagonal in  $\lo \widehat{\otimes} \so$. \et

\begin{proof}
(i) $\Longrightarrow$ (iii) This is \cite[Theorem 4.5]{GZ}.

(iii) $\Longrightarrow$ (ii) We note that from \cite{zhg2}, $\so$ is
(boundedly) approximately complemented in $\lo$. Hence the map
$\iota \otimes id_{\so} :  \so \widehat{\otimes} \so \longrightarrow
\lo \widehat{\otimes} \so$ is injective \cite{zhg1}. Therefore
$\iota \otimes id_{\so} $ maps a central approximate diagonal in
$\so \widehat{\otimes} \so$ into a central approximate diagonal in
$\lo \widehat{\otimes} \so$.

(ii) $\Longrightarrow$ (i) If $\so$ has a central approximate
diagonal in  $\lo \widehat{\otimes} \so$, then  a similar argument
to the proof of Theorem \ref{SOPseudoAmThm} gives a non-zero
function $f \in \lo$ such that $\delta_x *f = f$ $(x \in G)$. This
implies that $f$ is almost everywhere equal to a non-zero constant
and it follows that $G$ is compact.

(i) $\Longleftrightarrow$ (iv) If $G$ is compact, then $\lo
\widehat{\otimes} L^1(G^{op})=L^1(G \times G^{op})$ is biprojective
\cite[IV, Theorem 5.13]{Hel}. Hence $\so$ is a projective
$\lo$-bimodule since it can be regarded as a Banach left $\lo
\widehat{\otimes} L^1(G^{op})$-module \cite[IV, Theorem 5.3]{Hel}.

Conversely, suppose that $\so$ is a projective $\lo$-bimodule. Hence
there is a continuous $\lo$-bimodule morphism $\rho : \so
\longrightarrow \lo \widehat{\otimes} \so$ such that $\pi \circ
\rho=id_{\so}$. Let $1_G$ be the augmentation character of $L^1(G)$,
and put $$T = \iota \otimes 1_G: \lo \wot \so \rightarrow \lo : f
\otimes g \mapsto (\int_G g(s)ds ) f.$$ Now define the operator
$\rho_1 : \so \longrightarrow \lo$ by
$$ \rho_1=T\circ \rho.$$ It is easy to check that $\rho_1$ is a continuous
$\lo$-bimodule morphism. Moreover, for $f\in \so$ and $g\in I_0=\ker
1_G \cap \so$, we have
$$\rho_1(f*g)=\rho_1(f)\cdot g=\rho_1(f)1_G(g)=0.$$
Hence $\rho_1=0$ on $I_0$ since $\so I_0$ is dense in $I_0$.
Therefore $\rho_1$ induces a left $\lo$-module morphism
$$\tilde{\rho} : \so /I_0 \longrightarrow \lo .$$
However, $\so /I_0$ is isomorphic with $\C$ as a Banach $\lo$-module
for the product defined by
$$ f\cdot \lambda= \lambda \cdot f= 1_G(f)\lambda  \ \ \ (f\in \lo , \lambda\in \C).$$
Moreover, with the above identification, $1_G \circ
\tilde{\rho}=id_{\C}$. Thus $\C$ is a projective left
$\lo$-bimodule. This implies that $G$ is compact (see, for example,
\cite[Theorem 3.3.32(ii)]{Dal}.

%this one-dimensional module is denoted by $\C_\varphi$.

(i) $\Longleftrightarrow$ (v) If $G$ is compact, then $\so$ has a
central approximate identity $\{e_\alpha \}$ which has $L^1$-norm
equal to 1. On the other hand, from (iv), there is a continuous
$\so$-bimodule morphism $\theta : \so  \longrightarrow  \lo
\widehat{\otimes} \so$ which is the right inverse to the convolution
multiplication $\pi_1 :  \lo \widehat{\otimes} \so  \longrightarrow
\so$. Thus if we put $m_\alpha=\theta(e_\alpha)$, then it is
straightforward to show that $\{m_\alpha \}$ is a central
multiplier-bounded approximate diagonal in  $\lo \widehat{\otimes}
\so$ for $\so$. The converse follows easily because (v) implies
(iii).
\endpf\end{proof}

It is shown in \cite{GLau} that  every continuous derivation from a
symmetric Segal algebra $\so$ into $\so^*$ is approximately inner
whenever $G$ is an amenable group or a SIN group. The following
theorem is parallel to those results.

\bt \label{wstarapproxinner}
Let $G$ be a locally compact group, and let $\so$ be a symmetric
Segal algebra. Then for every continuous derivation $D : \so \to
\so^*$, $\pi^*\circ D$ is $w^*$-approximately inner. \et

\begin{proof}
Let $D : \so \to \so^*$ be a continuous derivation. Define the
operator $ \tilde{D} : \lo \to (\so \widehat{\otimes} \so)^*$ by
$$\l \D(f) \ , \ g\otimes h \r =\l D(f*g)-fD(g) \ , \ h \r \ \ (f\in \lo, g,h\in \so).$$
Since $D$ is a derivation, it is straightforward to verify that $\D$
is a continuous derivation. Let $\{e_\alpha\}_{\alpha \in I}$ be an
approximate identity in $\so$ having $L^1$-norm equal to 1. Define
the operator $\Lambda_\alpha : (\so \widehat{\otimes} \so)^* \to
L^\infty(G\times G)$ by
$$\Lambda_\alpha(T)(f\otimes g)=T(f*e_\alpha \otimes e_\alpha*g),$$
for every $T\in (\so \widehat{\otimes} \so)^*$ and $f,g\in \lo$.
Clearly each $\Lambda_\alpha$ is a continuous $\lo$-bimodule
morphism. Hence $\Lambda_\alpha \circ \D$ is a continuous derivation
from $\lo$ into $L^\infty(G\times G)$, and so, it is inner
(\cite[Theorem 5.6.41]{Dal}, in the case where $E=L^1(G\times G)$).
This means there is $\varphi_\alpha \in L^\infty(G\times G)$ such
that
$$ \Lambda_\alpha \circ \D=ad_{\varphi_\alpha} \ \ (\alpha \in I).$$
Let $\iota : \so \to \lo$ be the inclusion map and put
$\psi_\alpha=(\iota \otimes \iota)^*(\varphi_\alpha)$. Then
$$ (\iota \otimes \iota)^* \circ \Lambda_\alpha \circ \D=\ad_{\psi_\alpha} \ \ (\alpha \in I). \eqno{(1)}$$
However, since $\| e_\alpha \|_1=1$, it follows that for every $T\in
(\so \widehat{\otimes} \so)^*$ and $g,h\in \so$ \beqs | \l (\iota
\otimes \iota)^* \circ \Lambda_\alpha(T) \ , \ g\otimes h \r | & =&
| \l \Lambda_\alpha(T) \ , \ g\otimes h \r | \\ &=& | \l T \ , \
g*e_\alpha\otimes e_\alpha *h \r | \\ &\leq & \|T\|\| g*e_\alpha
\|_{\so}\| e_\alpha *h \|_{\so} \\ &\leq& \|T\|\| g\|_{\so}\|h
\|_{\so}. \eeqs Thus $ \| (\iota \otimes \iota)^* \circ
\Lambda_\alpha \| \leq 1$, and so,
$$ \| (\iota \otimes \iota)^* \circ \Lambda_\alpha \circ \D \| \leq 2 \|D\|.$$
Hence there is $\Delta \in B(\lo, (\so \otimes \so)^*)$ such that
$(\iota \widehat{\otimes} \iota)^* \circ \Lambda_\alpha \circ \D
\rightarrow \Delta$ in the $W^*OT$ of $B(\lo, (\so \widehat{\otimes}
\so)^*)$. Now take $f,g,h\in \so$. Then \beqs \l \Delta(f) \ , \
g\otimes h \r  &=& \lim_\alpha \l (\iota \otimes \iota)^* \circ
\Lambda_\alpha \circ \D (f) \ , \ g\otimes h \r \\ &
=& \lim_\alpha \l \D(f) \ , \ g*e_\alpha\otimes e_\alpha *h \r  \\
&= & \l \D(f) \ , \ g\otimes h \r \\ &=& \l D(f) \ , \ g*h \r \\ &=&
\l \pi^*\circ D(f) \ , \ g\otimes h \r. \eeqs Hence $\Delta \circ
\iota =\pi^*\circ D$. Therefore, from (1), it follows that
$$\pi^*\circ D=W^*OT-\lim_\alpha \ad_{\psi_\alpha}.$$
%f\cdot \psi_\alpha-\psi_\alpha \cdot f.$$
\endpf\end{proof}

\section{Approximate biflatness and pseudo-amenability of $\sa$}

In the preceding section we saw that the pseudo-amenablity of a
Segal algebra $\so$  in $\lo$, implies that $G$, and hence  $\lo$,
is amenable.  In this section we  prove that (operator) approximate
biflatness, and therefore pseudo-amenability,  of the (operator)
Segal algebra $\sa$ is much weaker than the (operator) amenability
of $A(G)$ (Theorems \ref{S1APseudoThm} and \ref{S1AOpPseudoThm}). On
the other hand,  the next theorem shows that the dual version of
Theorem \ref{CpctVsPseudoContrThm} is true.

If $F(G)$ is any collection of continuous functions on $G$, we let
$F_c(G)$ denote the set of compactly supported functions in $F(G)$.

 \blem \label{SA(G)PseudoContractLem}  Let $SA(G)$ be a Segal algebra in $A(G)$.  \bi
\item[(i)]  If $SA(G)$ has an approximate identity, then $SA_c(G)$
is dense in $SA(G)$.
\item[(ii)] If $G$ is discrete, then  the
indicator function at $g\in G$, $\delta_g$, belongs to $SA(G)$. \ei
\elem

\begin{proof} Let $u \in SA(G) $, $\epsilon >0$.   Take $e \in
SA(G)$  such that $\| ue - u \|_{SA} < \epsilon /2$. Choosing  $e_0
\in A_c(G)$ such that $\| e - e_0\|_A < \epsilon/ (2 \|u\|_{SA})$ we
have  $ue_0 \in SA_c(G)$ and
$$ \|ue_0 - u \|_{SA} \leq \| ue_0 - ue\|_{SA} + \|ue -u\|_{SA} \leq
\|u\|_{SA}\|e_0 - e\|_A + \epsilon/2 < \epsilon.$$ This proves (i).
If $G$ is discrete, then for $g \in G$, $\delta_g \in A(G)$ and we
can choose $u \in SA(G)$ such that $\| u - \delta_g\|_A < 1/2$. Then
$|u(g) - 1| < 1/2$, so $u(g) \neq 0$.  Now $\displaystyle \delta_g =
{1 \over u(g)} u \delta_g \in SA(G)$  proving statement  (ii).
\endpf\end{proof}

\bt \label{SA(G)PseudoContractThm} Let $SA(G)$ be an (operator)
Segal algebra of $A(G)$.   Then the following statements are
equivalent: \bi \item[(i)] $SA(G)$ has an approximate identity and
$G$ is discrete;   \item[(ii)]$SA(G)$ has an approximate identity
and  is (operator) approximately biprojective;
\item[(iii)] $SA(G)$ is (operator) pseudo-contractible. \ei  \et

\begin{proof} We prove the operator space version of the theorem.
Suppose that $G$ is discrete and that $SA(G)$ has an approximate
identity $(e_\lambda)_{\lambda \in \Lambda}$.  By Lemma
\ref{SA(G)PseudoContractLem}, we may assume that each $e_\lambda$
has compact support $E_\lambda$, and we can define $m_\lambda \in
SA(G) \wot_{op} SA(G)$ by $$ m_\lambda = \sum_{x \in E_\lambda}
e_\lambda(x) (\delta_x \otimes \delta_x) \qquad ( \lambda \in
\Lambda).$$ It is clear that $ a \c m_\lambda = m_\lambda \c a \ (a
\in SA(G))$ and $\pi (m_\lambda) = e_\lambda \ (\lambda \in
\Lambda)$, an approximate identity.  Hence, $SA(G)$ is operator
pseudo-contractible.

Assuming that  $SA(G)$ is operator pseudo-contractible, let
$(m_\alpha)_\alpha$ be an operator approximate diagonal for $SA(G)$
such that $a \c m_\alpha = m_\alpha \c a \ ( a \in SA(G))$. Let $T=
id_{SA(G)} \otimes \lambda(e) : SA(G) \wot_{op} SA(G) \rightarrow
SA(G)$ where $\lambda(e)$ is the (completely) bounded functional on
$SA(G)$ defined by $\lambda(e)u = u(e)$. By checking with elementary
tensors $m = u \otimes v$, one sees that  $T(a \c m) = a Tm$, $T(m
\c a ) = a(e) Tm$ $(a \in SA(G))$, and $Tm(e) = \pi(m)(e)$. Hence,
we can choose $\psi = Tm_\alpha$ such that $$ \psi = \psi a \quad (
a \in SA(G), \ a(e) = 1) \ \ {\rm and } \ \ \psi(e) \neq 0.$$ The
remainder of the proof is similar to the proof of \cite[Proposition
5]{Ren}.  Let $g \in G$ and choose $v \in A(G)$ such that $v(g) =0$,
$v(e) = 1$, and take $a \in SA(G)$ such that $a(e) = 1$. Then $av
\in SA(G)$ satisfies $av(e) =1$, so $ 0 = av\psi(g) = \psi(g)$.
Hence, $\displaystyle \delta_e = {1 \over \psi(e)} \psi$, which is a
continuous function on $G$. Hence, $G$ is discrete. The equivalence
of statements (ii) and (iii) is a special case of (the operator
space version of) \cite[Proposition 3.8]{GZ}.
\endpf\end{proof}

\begin{lem} \label{cbS1A(G)Lem} Let $F: \sah \longrightarrow \sa$ be a linear map with a
completely bounded extension $F^A: A(H) \longrightarrow A(G)$ and
(completely) bounded extension $F^L : L^1(H)  \longrightarrow \lo$.
Then $F$ is itself completely bounded.
\end{lem}
\begin{proof}  By definition, $\sa$ inherits its operator space
structure via the embedding
$$\sa \hookrightarrow A(G) \oplus_1 L^1(G): u \mapsto (u,u)$$
\cite[p.4]{FSW}. As $F^A$ and $F^L$ are completely bounded, so is
$$F^A \oplus F^L: A(H) \oplus_1 L^1(H) \rightarrow A(G) \oplus_1
L^1(G)$$ with $\|F^A \oplus F^L\|_{cb} \leq \|F^A\|_{cb}
+\|F^L\|_{cb}$.  Hence $F = (F^A \oplus F^L)\big{|}_{\sah} $ is also
completely bounded.
 \endpf\end{proof}

The ``completely bounded" part of the next lemma will not be needed,
but may be of independent interest.

\begin{lem} \label{S1AMultAILem} If $A(G)$ has an approximate
identity which is bounded in the (completely bounded) multiplier
norm, then so does $\sa$.
\end{lem}

\begin{proof}  Let $(e_\lambda)_{\lambda \in \Lambda}$ be an approximate identity for $A(G)$
with bound $R$ in the multiplier norm of $A(G)$; we may further
suppose that $(e_\lambda)$ is contained in $\sa$.   Given $x \in G$,
choose $v \in A(G)$ such that $\| v \NA = 1$ and $v(x) = 1$. Then
for any $\lambda$, $ |e_\lambda(x)| \leq \|e_\lambda v\|_\infty \leq
\|e_\lambda v\|_{A(G)}\leq R \| v\|_{A(G)} = R.$ Hence, \beq
\label{S1AMultAIPropEq1} \| e_\lambda\|_\infty \leq R \qquad (
\lambda \in \Lambda)\eeq  and therefore,  for any  $v \in \sa$, $$
\|e_\lambda v\NSA = \|e_\lambda v\NA +  \|e_\lambda v\NL  \leq R \|
v\NA +  R \| v\NL = R \|v \NSA.$$ Thus, $(e_\lambda)$ is also
bounded in the multiplier norm of $S^1A(G)$. As $(e_\lambda)$ is an
approximate identity for $A(G)$, $e_\lambda \rightarrow 1$ in the
topology of uniform convergence on compact subsets of $G$.  This,
together with equation (\ref{S1AMultAIPropEq1}), yields $$
\|e_\lambda v - v \NL \rightarrow 0 \qquad ( v \in \sa).$$
Consequently, $(e_\lambda)$ is an approximate identity for $\sa$.

Suppose now that $(e_\lambda)$ is bounded, again by $R$, in the
$cb$-multiplier norm in $A(G)$.  Again, we can suppose without loss
of generality that $(e_\lambda)$ is contained in $S^1A(G)$. From
equation (\ref{S1AMultAIPropEq1}) we know that the maps $$\lo
\longrightarrow \lo: a \mapsto e_\lambda a $$ are bounded by $R$. It
follows from Lemma \ref{cbS1A(G)Lem} that $(e_\lambda)$ is bounded
in the completely bounded multiplier norm taken with respect  to
$\sa$.
\endpf\end{proof}

If $u$ is a function defined on a subgroup $H$ of $G$, we let
$$\displaystyle{u^\circ(x) = \left \{ \begin{array}{ll}
                          u(x)  & \mbox{{\rm if} $x \in H $}\\
                          0 & \mbox{\rm otherwise.}
                            \end{array}
                 \right. } $$

\bt \label{ApproxBiflatHereditaryThm}  Let $G$ be a locally compact
group such that $\sa$ has an approximate identity. If $H$ is an open
subgroup of $G$, and $S^1A(H)$ is (operator) approximately biflat,
then so is $\sa$. \et

\begin{proof}Let $C$ be a transversal for left cosets of $H$ in
$G$, and assume that $e \in C$. Order the collection $\Gamma$ of
finite subsets of $C$ by inclusion, and for each $\gamma \in
\Gamma$, let $E_\gamma = \cup_{x \in \gamma} xH$.  Let $S^1A_\gamma
= \{ u \in \sa : u = u1_{E_\gamma} \}$, $S^1A_x= S^1A_{\{x\}}$.

 Assuming that Haar measure on $H$ is the restriction to
$H$ of the Haar measure on $G$, the map $u \mapsto u^\circ$ defines
a (completely) isometric isomorphism of $A( H)$ onto $A_e = \{ v \in
A(G) : v = v1_H\}$  \cite[Proposition 4.2]{Woo} and of $L^ 1(H)$
into $L^1(G)$. On $A(H)$, the inverse of this map is the restriction
of $r: A(G) \rightarrow A(H): u \rightarrow u\Big{|}_{H} $ to $A_e$,
which by \cite[Proposition 4.3]{Woo} is a complete contraction.  It
follows that $u \mapsto u^\circ$ is  an isometric isomorphism of
$S^1A(H)$ onto $S^1A_e$ which, by Lemma \ref{cbS1A(G)Lem}, is a
complete isomorphism. Similarly, for each $x \in G$ left translation
by $x^{-1}$   is a complete isomorphism of $S^1A_e$ onto $S^1A_x$
\cite[Lemma 4.4]{FW}. Hence, the (operator) approximate biflatness
of $S^1A(H)$ implies that of $S^1A_x$.  Lemma \ref{cbS1A(G)Lem} also
implies that the projection maps $S^1A(G) \rightarrow S^1A_x: u
\mapsto u1_{xH}$ are completely bounded, so $S^1A_\gamma =
\bigoplus_{x \in \gamma}S^1A_x \ (\gamma \in \Gamma)$ is (operator)
approximately biflat  by Proposition \ref{DirectSumProp}(b).  Let
$(e_\lambda)_\lambda$ be an approximate identity for $\sa$.  As
noted on page 10 of \cite{FSW}, $A_c(G)$ is dense in $S^1A(G)$, so
we may  assume  that each $e_\lambda$ has compact support so that
$(e_\lambda)_\lambda \subset \cup_{\gamma} S^1A_\gamma$. If we
define projections $P_\gamma$ of $S^1A(G)$  onto $S^1A_\gamma$ by
$P_\gamma u = u 1_{E_\gamma}$, the (operator) approximate biflatness
of $\sa$ follows from Proposition \ref{DirectedUnionProp}.
\endpf\end{proof}

\bt \label{S1APseudoThm}  Let $G$ be a locally compact group such
that $\sa$ has an approximate identity.   If $G$ contains an abelian
open subgroup, then $\sa$ is approximately biflat and therefore
pseudo-amenable. \et

\begin{proof}  Let $H$ be an open abelian subgroup of $G$.  Then
$A(H)$ is amenable and therefore biflat by \cite[Theorem
2.9.65]{Dal}. By Lemma \ref{S1AMultAILem}, $S^1A(H)$ has an
approximate identity $(e_\lambda)_\lambda$ which is bounded in the
multiplier norm of $S^1A(H)$; hence $(e_\lambda^2)_\lambda$  is also
an approximate identity for  $S^1A(H)$.  By first applying
Proposition \ref{AbstractApproxBiflatHeredProp}, then Theorem
\ref{ApproxBiflatHereditaryThm}, we can conclude that $\sa$ is
approximately biflat. Pseudo-amenability of $\sa$ follows from
Theorem \ref{ApproxBiflat-PseudoProp}.
\endpf\end{proof}

\bt \label{S1AOpPseudoThm} Let $G$ be a locally compact group such
that $\sa$ has an approximate identity, and suppose that  $G$
contains an open subgroup $H$ such that  $A(H)$ has an approximate
identity which is multiplier-norm bounded.  If $\Delta_H$ has a
bounded  approximate indicator, then $\sa$ is operator approximately
biflat and operator pseudo-amenable. \et

\begin{proof}   By \cite[Proposition 2.3]{ARS}, $A(H)$ is
operator biflat.  As with  the proof of Theorem \ref{S1APseudoThm},
Lemma \ref{S1AMultAILem}, Proposition
\ref{AbstractApproxBiflatHeredProp}, and Theorem
\ref{ApproxBiflatHereditaryThm} yield the operator approximate
biflatness of $\sa$.
\endpf\end{proof}

It is shown in \cite{ARS} that $\Delta_H$ has a bounded  approximate
indicator whenever $H$ can be continuously embedded in  a
[QSIN]-group. Every amenable group and every [SIN]-group is a
[QSIN]-group.  When $G_e$, the principle component of $G$,  is
amenable, the proof of \cite[Proposition 5.2]{GS} shows that $G$
contains an amenable open subgroup.  Hence we have the following
corollary to Theorem \ref{S1AOpPseudoThm}.

\bc \label{S1AOpBiflatCor} Let $G$ be a locally compact group such
that $\sa$ has an approximate identity. If $G_e$ is amenable, then
$\sa$ is operator approximately biflat and operator pseudo-amenable.
\ec

\br  \label{ApproxBiflatNotBiflatRem} \rm    The same arguments show
that under the hypotheses of Theorems \ref{S1APseudoThm} and
\ref{S1AOpPseudoThm}, $A(G)$ is, respectively,  approximately biflat
and operator approximately biflat. By choosing $G$ to be any
amenable group which contains an open abelian subgroup but which is
not a finite extension of an abelian group (such as the integer
Heisenberg group), $A(G)$ provides an example of a Banach algebra
which is approximately biflat, but not biflat.   Indeed, in this
case $A(G)$ has a bounded approximate identity, so if $A(G)$ were
biflat it would be amenable \cite[Theorem 2.9.65]{Dal} in
contradiction to the main result of \cite{FR}. \er

\section{Feichtinger's Segal algebra}
Let us recall the definition of $S_0(G)$.  Let $K$ be a compact
subset of $G$ with non-empty interior and $A_K(G)=\{u\in
A(G):\mathrm{supp}u\subset K\}$. We let
\[
q_K:\ell^1(G)\widehat{\otimes}_{op} A_K(G)\to A(G) \quad
q_K(\delta_s\otimes v)=s\!\ast\! v
\]
where $s\!\ast\! v(t)=v(s^{-1}t)$ and $\widehat{\otimes}_{op}$
denotes the operator projective tensor norm, which in this case is
the same as the projective tensor norm $\widehat{\otimes}$. Then we
set
\[
S_0(G)=\mathrm{ran} q_K
\]
and assign $S_0(G)$ the operator space structure (hence Banach space
structure) it inherits as a quotient of $\ell^1(G)\wot A_K(G)$. We
recall that this operator space structure is completely
isomorphically, though not completely isometrically, independent of
the choice of the set $K$. We do not know a tractable formula for
the norm of a matrix $[v_{ij}]$ in $M_n(S_0(G))$.  However, if we
consider a dual formulation, and consider matrices with a
``trace-class'' norm, $T_n(S_0(G))\cong T_n\widehat{\otimes}S_0(G)$, we
obtain for any $n\!\times\! n$ matrix $[v_{ij}]$ with entries in
$S_0(G)$
\[
\|[v_{ij}]\|_{T_n(\mathrm{ran} q_K)}
=\inf\left\{\sum_{k=1}^\infty\left\|[v_{ij}^{(k)}]\right\|_{T_n(A)}:
\begin{matrix} [v_{ij}]=\sum_{k=1}^\infty [s_k\!\ast\! v_{ij}^{(k)}] \\
\text{where each }s_k\in G\text{ and } \\
[v_{ij}^{(k)}]\in T_n(A_K(G))\end{matrix}\right\}.
\]
We recall for any operator space $\mathcal{V}$ that a linear map
$S:\mathcal{V}\to\mathcal{V}$ is completely bounded if and only if
the sequence of maps
\[
T_n(S):T_n(\mathcal{V})\to T_n(\mathcal{V}),\;
T_n(S)[v_{ij}]=[Sv_{ij}]
\]
are uniformly bounded, and we have
$\|S\|_{cb}=\sup_{n\in\mathbb{N}}\|T_n(S)\|$.

We let the {\it multiplier algebra} of $S_0(G)$ be given by
\[
MS_0(G)=\{u:G\to\mathbb{C}\; :\; uS_0(G)\subset S_0(G)\}.
\]
The usual closed graph theorem argument tells us that for each $u$
in $MS_0(G)$, the operator $v\mapsto uv$ is bounded.  We further
define the {\it completely bounded multiplier algebra} of $S_0(G)$
by
\[
M_{cb}S_0(G)=\{u\in MS_0(G)\; :\; v\mapsto uv:S_0(G)\to S_0(G)
\text{ is c.b}\}.
\]
We thus obtain the following modest description of the multipliers
and the completely bounded multipliers.

\bp \label{mult-Feitchtinger Segal algebra} Let $u:G\to\mathbb{C}$.

(i) $u\in MS_0(G)$ if and only if for any compact subset $K$
of $G$ with non-empy interior we have $uA_K(G)\subset A_K(G)$ and
\[
\|u\|_{M\mathrm{ran} q_K} =\sup\{\|u\,s\!\ast\! v\|_A:s\in G,v\in
A_K(G),\|v\|_A\leq 1\} <\infty.
\]

(ii) $u\in M_{cb}S_0(G)$ if and only if for any compact subset
$K$ of $G$ with non-empy interior we have $uA_K(G)\subset A_K(G)$
and
\[
\|u\|_{M_{cb}\mathrm{ran} q_K} =\sup\left\{\|[u\,s\!\ast\!
v_{ij}]\|_{T_n(A)}:
\begin{matrix}s\in G,[v_{ij}]\in T_n(A_K(G)) \\
\|[v_{ij}]\|_{T_n(A)}\leq 1\end{matrix} \right\}<\infty.
\]
\ep

We note that by regularity of $A(G)$, the condition $uA_K(G)\subset
A_K(G)$, for any $K$ as above, is equivalent to saying that $u$ is
locally an element of $A(G)$.

\medskip
\proof We will show only (ii), the proof of (i) being similar.

If $u\in M_{cb}S_0(G)$, let $m_u:S_0(G)\to S_0(G)$ be given by
$m_uv=uv$.  Note that for any $s$ in $ G$, compact $K\subset G$ with
non-empty interior and $[v_{ij}]$ in $T_n(A_K(G))$ we have
$[s\!\ast\! v_{ij}]\in T_n(S_0(G))$ with
\[
\|[s\!\ast\! v_{ij}]\|_{T_n(\mathrm{ran} q_K)}= \|[s\!\ast\!
v_{ij}]\|_{T_n(A)}.
\]
Since $A_K(G)\subset S_0(G)$, it is clear that $uA_K(G)\subset
A_K(G)$.  Moreover, since $S_0(G)$ is closed under translations, it
follows that $u(s\!\ast\! A_K(G))\subset s\!\ast\! A_K(G)$ too.
Hence for $s, [v_{ij}]$ as above with $\|[v_{ij}]\|_{T_n(A)}\leq 1$,
we have
\begin{align*}
\|[u\,s\!\ast\! v_{ij}]||_{T_n(A)}
&=\|[u\,s\!\ast\! v_{ij}]\|_{T_n(A)} \\
&=\|T_n(m_u)\|\leq\|m_u\|_{\mathcal{CB}(\mathrm{ran} q_K)}
\end{align*}
Conversely, if the latter conditions hold, we let $[v_{ij}]\in
T_n(S_0(G))$, $\varepsilon>0$, and find elements $s_k$ in $G$, and
matices $[v_{ij}^{(k)}]$ in $T_n(A_K(G))$ such that
\[
[v_{ij}]=\sum_{k=1}^\infty[s_k\!\ast\! v_{ij}^{(k)}]\text{ and }
\sum_{k=1}^\infty\|[v_{ij}^{(k)}]\|_{T_n(A)}<
\|[v_{ij}]\|_{T_n(\mathrm{ran} q_K)}+\varepsilon.
\]
Then we have
\begin{align*}
\|T_n(m_u)[v_{ij}]\|_{T_n(\mathrm{ran} q_K)}
&=\|[uv_{ij}]\|_{T_n(\mathrm{ran} q_K)} \\
&\leq\sum_{k=1}^\infty\|[u\,s_k\!\ast\! v_{ij}^{(k)}]\|_{T_n(A)} \\
&\leq\sum_{k=1}^\infty
\|u\|_{M_{cb}\mathrm{ran} q_K}\![s_k\!\ast\! v_{ij}^{(k)}]\|_{T_n(A)} \\
&<\|u\|_{M_{cb}\mathrm{ran} q_K}
\bigl(\|[v_{ij}]\|_{T_n(\mathrm{ran} q_K)}+\varepsilon\bigr).
\end{align*}
Hence for each $n$,
$\|T_n(m_u)\|\leq\|u\|_{M_{cb}\mathrm{ran}q_K}<\infty$, and thus
$u\in M_{cb}S_0(G)$.  \hfill $\square$\medskip

We let $MA(G)$ and $M_{cb}A(G)$ denote the algebras of multipliers
and completely bounded multipliers of $A(G)$.  The following is
immediate from the proposition above.

\bc {\bf (i)} $MA(G)\subset MS_0(G)$ with $\|u\|_{M\mathrm{ran}
q_K}\leq\|u\|_{MA}$ for any $u\in MA(G)$ and $K$ as above.

{\bf (ii)} $M_{cb}A(G)\subset M_{cb}S_0(G)$ with
$\|u\|_{M_{cb}\mathrm{ran} q_K}\leq\|u\|_{M_{cb}A}$ for any $u\in
M_{cb}A(G)$ and $K$ as above. In particular, $S_0(G)$ is a
completely contractive $B(G)$-module. \ec

\proof The only thing which does not follow directly from the
proposition above is that $S_0(G)$ is a completely contractive
$B(G)$-module. This can be seen by a straightforward modification of
the proof of the fact that $S_0(G)$ is a completely contractive
$A(G)$-module in \cite{Spr2}.  \hfill $\square$\medskip

We are now ready to state the main result of this section.

\bt \label{S_0biflatCor} Let $G$ be a locally compact group, and let
$H$ be an open subgroup of $G$ such that $H$ is weakly amenable and
$\Delta_H$ has a bounded  approximate indicator in $B(H\times H)$.
Then $\sf$ is operator biflat. In particular, $S_0(G)$ is operator
pseudo-amenable. \et

\begin{proof}
We first prove that $S_0(H)$ is operator biflat. Let
$\{f_\alpha\}_{\alpha\in I}$ be a bounded approximate indicator for
$\Delta_H$. For each $\alpha \in I$, define the operator
$\rho_\alpha : S_0(H\times H) \to S_0(H\times H)$ by
$$\rho_\alpha(u)=uf_\alpha \ \ \ (\alpha\in I).$$
By the preceding corollary, each $\rho_\alpha$ is a completely
bounded $B(H\times H)$-bimodule morphism. Moreover,
$$\|\rho_\alpha \|_{cb} \leq \|f_\alpha\|_{B(H\times H)} \leq M,$$
where $M=\sup \{ \|f_\alpha\|_{B(H\times H)}  \mid \alpha\in I \}$.
Let $\rho : S_0(H\times H) \to S_0(H\times H)^{**}$ be a
cluster-point of $\rho_\alpha$ in the $W^*OT$ of $\mathcal{CB}(S_0(H\times H),
S_0(H\times H)^{**})$. Clearly $\rho$ is a $B(H\times H)$-bimodule
morphism. Let
$$I(\Delta_H)=\{u\in S_0(H\times H) \mid u=0 \ \text{on}\  \Delta_H \};$$ and
$$I_0(\Delta_H)=\{u\in S_0(H\times H) \mid u\ \text{has a
compact support disjoint from}\ \Delta_H \}.$$ It is easy to see
that, for each $u\in I_0(\Delta_H)$, $uf_\alpha\rightarrow 0$ as
$\alpha\rightarrow \infty$. On the other hand, from Proposition
\ref{mult-Feitchtinger Segal algebra} and \cite[Theorem 3.1]{Spr2},
$S_0(H\times H)$ has an approximate identity bounded in its
cb-multiplier norm. Hence from the fact that $\Delta_H$ is a set of
synthesis for $A(H\times H)$ \cite[Theorem 3]{TT}, it follows that
$I_0(\Delta_H)$ is dense in $I(\Delta_H)$. Thus, for $u\in
I(\Delta_H)$ and $\epsilon >0$, there is $u_\epsilon \in
I_0(\Delta_H)$ such that $\|u-u_\epsilon \| < \epsilon$. Hence \beqs
\| uf_\alpha  \|  &\leq& \| (u-u_\epsilon)f_\alpha \|+\|u_\epsilon
f_\alpha \|  \\ & \leq& \| u-u_\epsilon \|M+ \|u_\epsilon f_\alpha
\|  \\ &\leq & \epsilon M+ \| u_\epsilon f_\alpha \| \\
&\rightarrow& \epsilon M, \eeqs as $\alpha \rightarrow \infty$. Thus
$uf_\alpha \to 0$ as $\alpha \to \infty$. This implies that $\rho=0$
on $I(\Delta_H)$. Hence
$$\tilde{\rho} : \frac{S_0(H\times H)}{I(\Delta_H)} \to S_0(H\times H)^{**} $$
is well-defined. Using the identification $S_0(H\times H) /
I(\Delta_H)=S_0(H)$ (see \cite[Theorem 3.3]{Spr2}), we can assume
that $\tilde{\rho}$ is defined on $S_0(H)$. It is clear that
$\tilde{\rho}$ is a continuous $B(H)$-bimodule morphism, and so, it
is a $S_0(H)$-bimodule morphism. Moreover, if $\pi : S_0(H\times H)
\to S_0(H)$ is the multiplication map, then $\pi^{**}\circ
\tilde{\rho}$ is the canonical embedding of $S_0(H)$ into
$S_0(H)^{**}$. Hence $S_0(H)$ is operator biflat.

Now by \cite[Corollary 2.6]{Spr2}, there is a natural completely
bounded algebra homomorphism from $S_0(G)$ onto $\ell^1(T)
\widehat{\otimes} S_0(H)$, where $T$ is a transversal for left
cosets of $H$ and $\ell^1(T)$ has pointwise multiplication. Hence,
by Proposition \ref{DirectSumProp}(iii), $S_0(G)$ is operator
biflat. Moreover, from Proposition  \ref{mult-Feitchtinger Segal
algebra}, $S_0(H)$ has an approximate identity bounded in its
cb-multiplier norm. Since the same is true for $\ell^1(T)$, it
follows that $S_0(G)$ has an approximate identity bounded in its
cb-multiplier norm. Hence, from Theorem
\ref{ApproxBiflat-PseudoProp}, $S_0(G)$ is operator pseudo-amenable.
\endpf\end{proof}

\noindent {\sc Department of Pure Mathematics, University of
Waterloo, Waterloo ON, N2L 3G1, Canada }

\noindent email: {\tt esamei@math.uwaterloo.ca}

\bigskip

 \noindent {\sc Department of Pure Mathematics, University of
Waterloo, Waterloo ON, N2L 3G1, Canada }

\noindent email: {\tt nspronk@math.uwaterloo.ca}

\bigskip

 \noindent {\sc Department of Mathematics and Statistics, University of Winnipeg,
515 Portage Avenue, Winnipeg MB, R3B 2E9  Canada }

\noindent email: {\tt r.stokke@uwinnipeg.ca}

\end{document}